\documentclass[11pt]{amsart}

\usepackage{amsmath,amssymb,amscd,amsfonts,verbatim}
\newtheorem{thm}{Theorem}[section]
\newtheorem{lem}[thm]{Lemma}

\newtheorem{prop}[thm]{Proposition}
\newtheorem{cor}[thm]{Corollary}

\newtheorem{assu-nota}[thm]{Assumption--Notation}
\newtheorem{ex}{Example}[section]
\theoremstyle{remark}
\newtheorem{remark}{Remark}
\newtheorem{assu}[remark]{Assumption}
\newcommand{\Q}{\mathbb Q}
\newcommand{\pp}{\mathbb P}
\newcommand{\OO}{\mathcal{O}}

\numberwithin{equation}{section}
\def\Qed{\hfill\raisebox{.6ex}{\framebox[2.5mm]{}}\\[.15in]}

\title [ Surfaces with $K^2=2\mathcal{X}-2$ and $p_g\geq 5$ ]
{ Surfaces with $K^2=2\mathcal{X}-2$ and $p_g\geq 5$
}
\author{ Mar{\'i}a Mart\'\i \ S{\'a}nchez}

\thanks{\parbox[t]{10cm}{Mar\'ia Mart\'i S\'anchez\\
Center for Mathematical Analysis, Geometry and Dynamical Systems\\
Instituto Superior T\'ecnico,
Universidade T{\'e}cnica de Lisboa\\
Av.~Rovisco Pais,
1049-001 Lisboa, PORTUGAL\\
Tel.: 0034-91-2219775\\
Fax: 0034-91-7545083\\
E-mail: mmartisanchez@educa.madrid.org
} \hfill }

\begin{document}

\begin{abstract} This note describes minimal surfaces $S$ of general type satis-
fying $p_g\geq 5$ and $K^2=2p_g$. For $p_g\geq 8$ the canonical map
of such surfaces is generically finite of degree $2$ and the bulk of the
paper is a complete
characterization of such surfaces with non birational canonical map. It
turns out that if $p_g\geq 13$, $S$ has always an (unique) genus 2
fibration, whose non 2-connected fibres can be characterized, whilst for
$p_g\leq 12$ there are two other classes of such surfaces with non
birational canonical map.
\medskip

\noindent \textbf{Keywords:}
\begin{tabular}[t]{l}
 Regular surfaces of general type\\
Canonical map \\
Involutions and double covers\\
Rational surfaces\\
Small  invariants\\
\end{tabular}

\

\noindent{\em 2000 Mathematics Subject Classification:} 14J29
\end{abstract}
\maketitle

\section{Introduction}

Noether's well known inequality states that a  minimal surface of general type satisfies $K^2 \geq 2\chi-6$. Surfaces with $K^2< 2\chi$ are always regular and Horikawa completely classified  minimal surfaces satisfying $K^2=2\chi-6,  2\chi-5$ and $2\chi-4$ (\cite{hoI}, \cite{ho4}, \cite{ho1}, \cite{hoIV}). Some aspects of surfaces with $K^2=2\chi-3$ have been studied by other authors (e.g. \cite{Xian}).

In this paper we characterize  minimal surfaces satisfying $K^2=2\chi-2$ and $p_g\geq 5$.
  Note that the case $p_g=4$ has been studied in \cite{ba-pig}.  Let us point out that with our methods we could also recover the classification of \cite{ba-pig}.

We start  with an overview of the case.  From the results of \cite{ho1} the canonical map is not composed with a pencil. Also, by \cite{men-par}, the canonical map has  always degree $\leq 2$. If the canonical map is birational, then $p_g \leq 7$.
The bulk of our analysis is the case when the canonical map has degree 2.   In this   case  the canonical image is always a  rational surface and  we consider the number  $t$ of isolated  fixed points of the involution induced by the canonical map.   The main results obtained are:

\begin{itemize} \item {\it Theorem \ref{class}}: Let $S$ be a minimal surface with $K_S^2=2\chi-2$ and $p_g\geq 5$. Then $S$ satisfies exactly one of the following:
\begin{enumerate}
\item [(I)] the canonical map $\phi_{K_S}$ is birational and
    \begin{enumerate}
    \item[(Ia)] $|K_S|$ is free from base points and $p_g\leq 7$, or
    \item[(Ib)] $|K_S|$ has exactly one (simple) base point and $p_g=5$;
    \end{enumerate}
\item [(II)] the canonical map factors through an involution $i$ and the number $t$ of isolated fixed points of $i$ is:
    \begin{enumerate}
    \item[(IIa)] $t=0$, or
    \item[(IIb)] $t=2$, or
    \item[(IIc)] $t=4$.
    \end{enumerate}
\end{enumerate}

 Furthermore

    \item {\it Proposition \ref{thm_t=0}}: If $t=0$, then $p_g\leq12$ and
 $S$ is the minimal resolution of a double cover of $\mathbb{F}_r$, $r\leq 3$, branched on a curve in $|8C_0 + 2(5+2r)f|$  having $12-p_g$ singular points of multiplicity 4 as only essential singularities.
    \item {\it Theorem \ref{thm_t=2}}: If $t=2$, then $p_g\leq 8$ and one of the following occurs:
        \begin{enumerate}
            \item $S$ is the minimal resolution of a double cover of a {\it weak Del Pezzo} surface $T$ of degree $p_g+1$ branched on an effective divisor  in $|-4K_T|$ having exactly two (3,3)-points as essential singularities.

            \item  $S$ is the minimal resolution of a double cover of $\mathbb{F}_r$, $r\leq 2$,
               whose branch curve is the union of a curve in $|8C_0 + (9+4r)f)|$ with a fibre. The curve has $8-p_g$ singular points of multiplicity 4 and another of type $(4,4)$ and the fibre is tangent to the curve at the $(4,4)$-point.
        \end{enumerate}
    \item If $t=4$, the surface has a unique genus 2 pencil and in {\it Proposition \ref{prop-genus2}}  we see the different  possibilities for the  singularities of the branch locus.

\end{itemize}

All these types of surfaces, except possibly type $Ib$, do exist. For surfaces of type $I$ we refer to Remark \ref{type_I}  and Proposition \ref{type_Ib}.
Surfaces of type $IIa, IIb$ and $IIc$ are easily seen to exist using the descriptions as double covers given in Proposition \ref{thm_t=0}, Theorem \ref{thm_t=2} and Proposition \ref{prop-genus2}.

\bigskip The paper is organized as follows. In Section 2 some general properties of involutions are recalled. In Section 3  the canonical map of these surfaces is studied,  yielding a first division into cases. In the remaining sections each of these cases is described.

\

\noindent {\it Notation.}
 We work over the complex numbers. All varieties are projective algebraic. All the notation we use is standard in algebraic geometry.
We just recall the definition of the numerical invariants of a smooth surface $X$:  the self-intersection number $K^2_X$ of the canonical divisor $K_X$, the {\em geometric genus} $p_g(X):=h^0(K_X)=h^2(\OO_X)$, the {\em irregularity} $q(X):=h^0(\Omega^1_X)=h^1(\OO_X)$ and the {\em holomorphic Euler characteristic} $\chi(X):=1+p_g(X)-q(X)$.


An {\em involution} of a surface $S$ is an
automorphism of $S$ of order 2.
We say that a curve
singularity is {\em nonessential} if it is either a double point or a
triple point which resolves to at most a double point after one
blow-up.
Other curve singularities are said to be {\em essential}.
A $(m,k)-point$ of a curve is a point of multiplicity $m$, which resolves to an ordinary point of multiplicity $k$ after one blow-up.
We say that a map is {\em composed with
an involution} $i$ of $S$ if it factors through the double cover $S\rightarrow
S/i.$

We do not distinguish between line bundles and divisors on a smooth
variety. Linear equivalence is denoted by $\equiv$ and numerical equivalence by $\sim$.


\

\noindent {\it Acknowledgments.}

I would like to thank prof. M. Mendes
Lopes for her constant help in this work. Also, I wish to thank the referee for his/her careful reading and for
all the suggestions that improved substantially the presentation of this
paper.

The author is a collaborator of the Center for Mathematical Analysis, Geometry
and Dynamical Systems of Instituto Superior T´ecnico, Universidade
T´ecnica de Lisboa and was partially supported by FCT (Portugal) through program POCTI/FEDER,  Project
POCTI/MAT/44068/2002 and the doctoral grant SFRH/BD/17596/2004.

\

\section{Involutions on surfaces}

\

Let $S$ be a minimal surface of general type. Given an involution $i$ on $S$, its fixed locus is the union of a smooth curve $R$ (possibly empty) and of $t\geq 0$ isolated points $P_1,...,P_t$.
Let $\pi ': S \to S/i$ be the quotient map and set $B'':=\pi'(R)$. The surface $S/i$ is normal and $Q_1:=\pi '(P_1), ..., Q_t:=\pi '(P_t)$ are ordinary double points, which are the only singularities of $S/i$. Resolving these singularities we get a commutative diagram
\begin{equation}\label{diagrama_inv}
\begin{CD}\ V@>h>>S\\ @V\pi VV  @VV \pi' V\\ W@>g >> S/i
\end{CD}
\end{equation}
where $g$ is the minimal desingularization map,  $h$ is the blow up of $S$ at $P_1,...,P_t$ and $V$ is obtained by base change and normalization. Notice that  the curves $A_i:=g^{-1}(Q_i)$ are $(-2)$-curves. Setting $B:=g^*(B'')$, $\pi$ is a double cover whose branch locus $B'$ is given by:
$$2L\equiv B':= B + \sum _1^t A_i.$$

We recall the well known formulas (cf. \cite{bpv}, Chapter V, Section 22):
\begin{equation}\label{eq-involutions1}
K_S^2 -t=K_V^2=2(K_W+L)^2,
\end{equation}
\begin{equation}\label{eq-involutions2}
\mathcal{X}(\mathcal{O}_S)=\mathcal{X}(\mathcal{O}_V)=2\mathcal{X}(\mathcal{O}_W) + \frac{1}{2}L(K_W+L).
\end{equation}

Since   $\pi^*(2K_W+B)=h^*(2K_S)$ and $S$ is a minimal surface of general type, $2K_W+B$ is a nef and big divisor and (see
\cite{ca-ci-men}, \cite{men-par2}, also (\cite{ri1})):

\begin{equation}\label{lema_involution}
 (2K_W + B)^2=2K_S^2,
 \end{equation}
\begin{equation}\label{lema_involutioni}
K_W(K_W+L)=\frac{1}{2}K_W(2K_W+B)= \frac{1}{2}(K_S^2-t)  - 2\mathcal{X}(\mathcal{O}_S)+4\mathcal{X}(\mathcal{O}_W),
 \end{equation}
\begin{equation}\label{lema_involutioniii}
h^1(2K_W+L)=h^2(2K_W+L)= 0,
\end{equation}
\begin{equation}\label{eq-involutions-rito}
 t= K_S^2 + 6\mathcal{X}(\mathcal{O}_W) -2\mathcal{X}(\mathcal{O}_S)-2h^0(W,\mathcal{O}_W(2K_W + L)),
 \end{equation}

Note that  the bicanonical map of $S$ factors through $i$ if and only if $h^0(W,\mathcal{O}_W(2K_W + L))=0$ (see e.g. proof of Proposition 2.1 of \cite{men-par2}).

\medskip

It will be important in what follows to study the divisor $3K_W+B$. This divisor is not necessarily nef but,
as shown in  Proposition 3.9 of \cite{ca-ci-men}, it is possible to assume it is nef.

  More precisely, from this proposition and its proof we obtain:

\begin{prop}[\rm \cite{ca-ci-men}]\label{prop_minimalmodel}
Suppose $h^0(3K_W+B)\neq 0$. There exists a birational morphism $p:W \to P$ where $P$ is a smooth surface and an effective divisor $\bar B$ on $P$ with the following properties:
\begin{itemize}
\item there are $t$ $(-2)$-curves $C_i$ on $P$ such that $p^*(C_i)=A_i$, $i=1,...,t$ and $\bar B$ is disjoint from the union of the curves $C_i$;
\item  there is  $\bar L$ in Pic $(P)$ such that $\bar B+\sum_1^t C_i=2\bar L=\bar{B}'$ and $p^*(K_P+\bar L)=K_W+L $;
\item the double cover $\bar V$ of $P$   defined by $\bar B+\sum_1^t  C_i=2\bar L$ is a surface with at most Du Val singularities and such that
 $V$ is the minimal desingularization of $\bar V$;
\item $p^*(2K_P+\bar B)=2K_W+B$;
\item  $3K_P + \bar B$ is nef.
\end{itemize}
\end{prop}
\begin{proof}
All the above follows easily from the statement and the proof  of Proposition 3.9 of \cite{ca-ci-men}.
We  remark that although  Proposition 3.9 of \cite{ca-ci-men} is stated for surfaces $S$ with $p_g=0$,  the proof is valid for any double cover as above.
\end{proof}

\begin{remark}\label{formulasP} It is easily seen that the formulas and properties given above will still hold if we substitute $W$ with $P$, $B$ with $\bar B$ and $L$ with $\bar L$. Note also that, for $m\geq 3$,  $mK_W+B=p^*(mK_P+\bar B)+ (m-2)E$, where $E$ is the exceptional divisor of $p$.  Hence  $|mK_W+B|=p^*(|mK_P+\bar B|)+(m-2)E$ and thus in particular  $|mK_W+B|$ is composed with a pencil if and only if  $|mK_P+\bar B|$  is.
\end{remark}
\begin{remark}\label{inv_t=0}{\rm  In the case where the involution $i$ has no isolated fixed points and $H^0(W,2K_W+L)\neq 0$, we notice that also $2K_P+\bar L$ will be nef.  In fact in that case, $V=S$ and $\pi^*(K_W+L)=K_S$, yielding in particular that $K_W+ L$ is nef and thus also that $K_P+\bar L$ is nef.  Suppose, for contradiction, that  $2K_P+\bar L$ is not nef.
Then there exists an irreducible curve $E$   such that $(2K_P+\bar L)E<0$. Since $2K_P + L$ is a nonzero effective divisor, one must have that $E^2<0$ (see \cite{reid0}, pg 34). On the other hand, since $K_P+L$ is nef, necessarily  $K_P E<0$. So $E$ is a $(-1)$-curve. Then from $E(2K_P+\bar{L})< 0$, we conclude $E(K_P+\bar L)=0$ and so $E(3K_P+\bar B)<0$. This contradicts the construction of $P$ as above. So $2K_P+\bar{L}$ is nef.}
\end{remark}

\

\section{The canonical map}

\

In this section $S$ denotes a minimal surface of general type with $K_S^2=2\mathcal{X} -2$ and $p_g\geq 5$. First:

\begin{lem}
Let $S$ be a minimal surface of general type with $K_S^2=2\chi-2$ and $p_g\geq 5$. Then:
\begin{itemize}
 \item S is regular;
 \item $K_S^2=2p_g$;
 \item $S$ has no torsion divisors.
 \end{itemize}
\end{lem}
\begin{proof}
From (\cite{bom}, Lemma 14), one has that $q=0$ and thus $K_S^2=2p_g$. Also, by (\cite{cmtorsion}, Theorem A), surfaces satisfying $K_2=2\chi -2$ can only have torsion if $p_g\leq 4$.
\end{proof}

We write $|K_S|=|H| + F$, where $|H|$ is the moving part of $|K_S|$ and $F$ the fixed part and
let $\rho: \tilde{S} \to S$ be a composition of blow-ups such that the variable part
$| \tilde{H}|$ of $|\rho^*K_S|$ is free from base points.  We denote by $E$ the exceptional divisor of $\rho$, by $F'$ the fixed part of $|\rho^*K_S|$ and by $\Sigma$ the canonical image of $S$.

\begin{lem}\label{lem_clasification}
Let $S$ be a minimal surface of general type with $K_S^2=2\chi-2$ and $p_g\geq 5$. Then one of the following occurs:
\begin{itemize}
\item $\phi_{K_S}$ is a birational map;
\item $\phi_{K_S}$ is a rational map of degree 2 onto a rational  surface.
\end{itemize}
\end{lem}
\begin{proof} Note first that $\phi_{K_S}$ is generically finite because from (\cite{ho1}, Theorem 1.2) surfaces with $p_g\geq 5$ and $|K|$ composed with a pencil must satisfy $K^2> 4p_g-7$. So,
the canonical image $\Sigma$ is an irreducible and nondegenerate surface in $\pp^{p_g-1}$,
$$2p_g=K_S^2\geq (deg \, \phi_{K_S})(deg \, \Sigma)\geq (deg \, \phi_{K_S})(p_g-2).$$
 So $deg \, \phi_{K_S}\leq 3$ and if $deg \, \phi_{K_S}=3$, then $p_g=5$ or $p_g=6$.

Assume that $deg \, \phi_{K_S}=3$. Then a general curve in $|K_S|$ is smooth, because $|K_S|$ has no base points if $p_g=6$, or a unique base point if $p_g=5$. Since $\Sigma$ is a surface of minimal degree $p_g-2$ in $\pp^{p_g-1}$ we obtain a contradiction to (\cite{men-par}, Theorem 2.1) where it is shown that if the general curve in $|K_S|$ is smooth and the canonical map is of degree 3 onto a surface of minimal degree $p_g-2$ in $\pp^{p_g-1}$, then $p_g\leq 5$ and $K_S^2\leq 9$.
Therefore, we have $deg \, \phi_{K_S}\leq 2$.

If  $deg \, \phi_{K_S}= 2$, then $\Sigma$ is a surface of degree $\leq p_g$ in $\pp^{p_g-1}$. From $p_g\geq 5$, one obtains  $deg \, \Sigma\leq p_g< 2p_g-4$ and thus, by  (\cite{beauville2}, Lemme 1.4), $\Sigma$ has Kodaira dimension $-\infty$.   Since $S$ is regular we conclude that $\Sigma$ is a rational surface.

\end{proof}


If $\phi_{K_S}$ has degree $2$, then the canonical map factors through an involution $i$. In this case, we recall the diagram (\ref{diagrama_inv})
$$
\begin{CD}\ V@>h>>S\\ @V\pi VV  @VV \pi' V\\ W@>g >> S/i
\end{CD}
$$

where $\pi$ is a double cover with branch locus $2L\equiv B'= B + \sum _1^t A_i$. We consider the $\Q$-divisor $B/2$ and we keep the notation of Section 2.
\begin{remark}\label{lem_inv}
{\rm Note that  if the canonical map factors through an involution, since $\mathcal{X}(S/i)=1$ by Lemma \ref{lem_clasification},
we have
$$t=4-2h^0(2K_W +L)$$
so the number of isolated fixed points of the involution is $t=0, 2$ or $4$. }
\end{remark}

Also, as a consequence of the double cover formulas
\begin{prop}\label{numeros2} Let $S$ be a minimal surface of general type with $K_S^2=2\chi-2$ such that  the canonical map factors through an involution with rational quotient. Then, with the notation above:
\begin{enumerate}
\item $K_W(K_W+L)=-p_g+2-\frac{t}{2}$;
\item $h^0(2K_W +L)=2-t/2$;
\item $(2K_W + L)^2=K_W^2 -p_g - \frac{3}{2}t +4$;
\item $(2K_W+B)^2=4p_g$;
\item $(2K_W+B)(2K_W+L)=4\epsilon+4-t$;
\item $K_WL=K_W(B/2)=2-p_g -t/2 - K_W^2$;
\item $(B/2)^2=L^2+\frac{t}{2}=3p_g + t + K_W^2 - 4;$
\item $p_a(2K_W+B)=p_g+3 -\frac{t}{2};$
\item $h^0(3K_W+B)=p_a(2K_W+B);$
\item if $3K_W+B$ is nef and big then $h^0(4K_W+B)=p_a(3K_W+B).$
\end{enumerate}
\end{prop}
\begin{proof}
All the equalities above, except the two last, are a direct consequence of the formulas in the previous section. In fact, since $S$ is regular, it satisfies $\chi(\mathcal{O}_S)=p_g +1$ and  $W$ being a rational surface satisfies $\chi(\mathcal{O}_W)=1$.

 The two last equalities come from  the Riemann-Roch theorem  since $2K_W+B$ is nef and big.
 \end{proof}

Let us recall the next result due to Castelnuovo ($cf$. \cite{acgh}):
\begin{lem}[{\em Castelnuovo's Bound}]\label{castelnuovo}
Let $C$ be a smooth curve of genus g that admits a birational mapping onto a
nondegenerate curve of degree d in $\pp^r$. Let $m=[ (d-1)/(r-1)]$ and
$\varepsilon =(d-1)-m(r-1).$ Then $g\leq \pi (d,r)$, where
$\pi (d,r)=m(m-1)(r-1)/2 + m\varepsilon$.
\end{lem}

We will need also:
\begin{lem}\label{prop_sbp}
Let $S$ be a minimal surface of general type such that $K_S^2=3p_g-5$ and
$q=0$.
If the canonical map is birational, then $|K_S|$ does not have a fixed
component and has at most one (simple) base point.
\end{lem}
\begin{proof}
We use the notation introduced in  the beginning of this section. Since we are assuming $\phi_{K_S}$ birational, by \cite{ho4},  $\tilde{H}^2\geq 3p_g-7$.
Thus we have
$$\tilde{H}^2=3p_g-5, \ 3p_g-6, \ \text{or} \ 3p_g-7.$$
If $\tilde{H}^2=3p_g-5$, the $|K_S|$ has no base points. If $\tilde{H}^2=3p_g-6$, since $K_S=H+F$ and $K_S$ is nef and 2-connected, then $F=0$ and we have at most one (simple) base point. So, if the statement is
not true, then
necessarily  $\tilde{H}^2=3p_g-7$ and the general curve $C$ in
$|\tilde{H}|$
is nonsingular. Note that $|\rho^*K_S|= |\tilde{H}| + E + F',$ where $E$ is the exceptional divisor of $\rho$ and $F'$ the fixed part of $\rho^*K_S$, then
$\tilde{H}(E
+ F')> 0$.
 So, by the adjunction formula, $C$ is of genus
$$3p_g -6 +\frac{1}{2}\tilde{H}(2E +F'),$$
On the other hand, from  Lemma \ref{castelnuovo} we obtain $g(C)\leq 3p_g-6$, a
contradiction.
\end{proof}

We can now give a  rough classification:

\begin{thm}\label{class}
Let $S$ be a minimal surface with $K_S^2=2\chi-2$ and $p_g\geq 5$. Then $S$ satisfies exactly one of the following:
\begin{enumerate}
\item [(I)] the canonical map $\phi_{K_S}$ is birational and
    \begin{enumerate}
    \item[(Ia)] $|K_S|$ is free from base points and $p_g\leq 7$, or
    \item[(Ib)] $|K_S|$ has exactly one (simple) base point and $p_g=5$;
    \end{enumerate}
\item [(II)] the canonical map factors through an involution $i$ and the number $t$ of isolated fixed points of $i$ is:
    \begin{enumerate}
    \item[(IIa)] $t=0$, or
    \item[(IIb)] $t=2$, or
    \item[(IIc)] $t=4$.
    \end{enumerate}
\end{enumerate}
\end{thm}
\begin{proof}
If $deg \, \phi_{K_S}= 1$, by the Castelnuovo inequality, one has $p_g\leq 7$. Now, using (\cite{konno}, Lemma 1.3) and (\cite{a-kon}, Lemma 1.1) for $p_g=6$ and $7$ respectively, we obtain that $|K_S|$ is free from base points. Finally,
by Lemma \ref{prop_sbp} for $p_g=5$,  the canonical system $|K_S|$ does not have fixed components and has at most one (simple) base point.

If $deg \, \phi_{K_S}= 2$, the result follows from Remark \ref{lem_inv}.
\end{proof}

\begin{remark}\label{type_I}
{\rm  Surfaces of type $(Ia)$ with $p_g=7$, have been studied, among others, by Ashikaga and Konno (\cite{a-kon}) and Miranda (\cite{mi}). In particular, Miranda has proved that $\phi_{K_S}$ maps $S$ into the Veronese cone or into a rational normal scroll.

For surfaces of type $(Ia)$ with $p_g=6$, we refer to \cite{konno}. For this case Konno has shown that  the canonical image is contained in a threefold $W$ of $\Delta$-genus $\leq 1$ which is cut out by all quadrics through the canonical image.

Ciliberto in \cite{ci}, proves the existence of surfaces of type $(I_a)$ with $p_g=5$. He has studied the moduli space of such surfaces, its
dimension and its
unirationality. Furthermore he has shown that the canonical image of a generic
such surface has
only isolated singularities and cannot lie in a quadric.

Surfaces of type $(Ib)$, as far as we know,  have not  been  studied yet and we do not know whether they exist.  Next we give some of their properties.}
\end{remark}

\

\begin{prop}\label{type_Ib}
If $S$ is a surface of type $(Ib)$, then:
\begin{itemize}
\item a general canonical curve $D$ is smooth and non hyperelliptic of genus 11 ;
\item the image of $D$ via the canonical map of $S$ is a curve $D_0$ of degree 9 in $\pp^3$ with one double point;
\item the canonical image of $S$ is contained in a singular quadric of $\pp^4$.
\end{itemize}

\end{prop}

\begin{proof}
 Since by Theorem \ref{class}  $|K_S|$ has only one simple  base point,  a general member $D \in |K_S|$ is irreducible and nonsingular. Thus, from $K^2_S=10$ and the adjunction formula we obtain that the geometrical genus of $D$, $g(D)$, is 11. Also $D$ is nonhyperelliptic because the canonical map of $S$ is birational. The image  $D_0$  of $D$ by $\phi_{K_S}$ is an irreducible nondegenerate curve of degree $9$ in $\pp^3$. Since $g(D)=11$, $11\leq p_a(D_0)$. On the other hand, applying the main theorem of  \cite{ccg} we obtain  $p_a(D_0)\leq 12$.
 If $p_a(D_0)=11$, $D_0$ is a nonsingular curve with degree 9 and $g=11$ in $\pp^3$, but this is not possible, (cf.(\cite{har}, Exercise 6.4 and Remark 6.4.1). Hence $p_a(D_0)=12$ and so $D_0$ has exactly one singular double point.

For the last assertion consider the subspace $V$ of $H^0(2K_S)$ generated by products of sections of $K_S$. We claim that $V$ has dimension 14. By \cite[Prop. 3.1]{debarre}
  dim $V\geq 14$. Assume for contradiction that dim $V\geq 15$. Since the kernel of the restriction map $H^0(S, 2K_S) \to H^0(D, K_D)$ is isomorphic to $H^0(S,K_S)$ and so 5-dimensional, the image of the restriction of $V$ to $D$ is at least 10-dimensional. Let $x$ be the unique base point of $|K_S|$, then every section in $V$ vanishes at least twice in $x$. So we conclude that $h^0(D, K_D-2x) \geq 10$ and hence, by the Riemann-Roch theorem and $g(D)=11$, $h^0(D, 2x)\geq 2$, a contradiction because $D$ is nonhyperelliptic.
  So dim $V =14$ and since $h^0(\pp^4, \mathcal{O}_{\pp^4}(2))=15$ we conclude that the canonical image $\Sigma$ of $S$ is contained in a quadric $Q$ of $\pp^4$. Furthermore, $Q$ must be singular because the
degree of $\Sigma$ is 9 and any surface on a nonsingular quadric of
$\pp^4$ is a complete intersection.

\end{proof}

\

\section{Surfaces of type $(IIa)$}

\

We start by stating a general fact:
\begin{lem}\label{lem-marga}
Let $D$ be a nef divisor on a rational surface such that $D^2=0$ and
$KD=-2r<
0$. Then $|D|=|rN|$, where $|N|$ is a base point free pencil of curves of genus
$0$.
\end{lem}
\begin{proof}
By the Riemann-Roch theorem,  $h^0(D)\geq r+1$.

Write $|D|= |M| + Z$, where $Z$ is the fixed part and $|M|$ is the moving
part. Since $D$ is nef and $M$, being the moving part of
$|D|$
is nef,
$0=D^2\geq DM = M^2 + MZ \geq MZ\geq 0.$

So $M^2=MZ=Z^2=0$.  In particular $|M|$ is composed with a pencil, hence,
$M=aN$
with $h^0(M)=a+1\geq r+1.$   Since $MZ=0$,  $NZ=0$ and either $Z=0$ or so by
Zariski's lemma (see Chp. III, Lemma (8.2) of \cite{bpv}) $Z=b N$, with $b$ a positive rational
number.  From $KD=-2r$, we conclude that $KN<0$ and so, by adjunction,
$KN=-2$
and $KM\geq KD$.   Since $-2a=KM\leq KD=-2r$, $Z$ must be zero and $a=r$,
i.e.
$h^0(M)=r+1$.
\end{proof}

Throughout this section, we make the following assumption:
\begin{assu}\label{assum_t=0}
 $S$  is a minimal surface with $K_S^2=2p_g$, $q=0$, $p_g\geq 5$, and such that $\phi_{K_S}$ has degree 2 with $t=0$.
\end{assu}

We keep the notation of Section 2, and in particular $p:S/i \to P$ is the birational morphism such that $2K_P+\bar{L}$ is nef (see Remark \ref{inv_t=0}). Then:

\begin{lem}\label{prop_marri}
Let $S$ be as in Assumption \ref{assum_t=0}, then the linear system
 $|2K_P + \bar{L}|$ is a rational pencil without base points. Moreover, $|\pi^*(2K_P+\bar{L})|$ is a hyperelliptic pencil of genus 3 in $S$.
\end{lem}
\begin{proof}
Remark first that $(2K_P+\bar{L})^2= \alpha \geq 0$, since $2K_P+\bar{L}$ is nef.  Next, by Proposition \ref{numeros2}, we get $(K_P+\bar{L})^2=p_g$ and $(K_P+\bar{L})(2K_P +\bar{L})=2$. It follows immediately that $(K_P+\bar{L} + 2K_P+\bar{L})^2=p_g+ \alpha + 4>0$, so as a consequence of the Index theorem $p_g\cdot \alpha \leq 4$ and hence $(2K_P+\bar{L})^2=0$. Also, from Proposition \ref{numeros2} we see $(2K_P+\bar{L})K_P=-2$. So, as a consequence of Lemma \ref{lem-marga}, we conclude that  $|2K_P+\bar{L}|$ is a base point free pencil of rational curves and, since $(2K_P+\bar{L})(K_P+\bar{L})=2$, we finish the proof.
\end{proof}

\begin{prop}\label{thm_t=0}
Let $S$ be a surface as in Assumption \ref{assum_t=0}. Then $p_g\leq12$ and
 $S$ is the minimal resolution of a double cover of $\mathbb{F}_r$, $r\leq 3$, branched on a curve in $|8C_0 + 2(5+2r)f|$  having $12-p_g$ singular points (possibly infinitely near) of multiplicity 4 as only essential singularities.
\end{prop}
\begin{proof}
Let $S \to P$ be the map of degree 2 with  branch curve $\bar{B}=2\bar{L}$, with possibly inessential singularities. Using Lemma \ref{prop_marri}, we know that
$|2K_P + \bar{L}|$ is a genus 0 pencil without base points. Then, we have $P\neq \pp^2$ and, from Proposition \ref{numeros2}, $K_P^2=p_g-4$. So contracting $12-p_g$ $(-1)$-curves contained in the fibres of $|2K_P + \bar{L}|$ we get a birational morphism $\gamma: P \to \mathbb{F}_r.$

Let $f$ be a fibre of $\mathbb{F}_r$ and $C_0$ a section with $C_0^2=-r$. Denote $\bar{L}= \gamma^*(aC_0 + bf)- \sum c_iE_i$, since $2K_P+\bar{L} = \gamma^*(f)$, then $$(a-4)C_0 + (b-2(2+r))f + \sum (2-c_i)E_i=f$$
and we obtain $a=4$, $c_i=2$ and $b=+5 + 2r$.

Note that $c_i=2$ for every $i$ means that the essential singularities are quadruple points.

In the end, we can write  $\gamma^*(C_0)= B_0  + \sum
\xi_iE_i$, where $B_0$ is the strict transform of $C_0$.  Then  $0\leq (K_P+ \bar{L})B_0=(2C_0 +(3+r)f)C_0 -\Sigma\xi_i \leq 3-r$
which implies $r\leq 3$.
\end{proof}

\begin{remark}\label{rem_r=3}
In Proposition \ref{thm_t=0}, if $p_g \leq 11$ we have also that $r\leq 2$. Indeed, if $r=3$ the image of $\bar{B}$ is in $| 8C_0 + 22f|$. We see that $C_0$ is in the fixed part of $| 8C_0 + 22f|$ and $C_0(7C_0 + 22f)=1$, so the essential singularities are not contained in $C_0$. Since $p_g \leq 11$ there exists at least one essential singularity on a fibre of the ruling. Blowing up this point and next contracting the strict transform of the fibre, we obtain a new birational morphism from $P$ onto $\mathbb{F}_2$ with a quadruple point on the infinity section.

\end{remark}

\begin{cor}\label{cor_t=0}
Let $S$ be as in Proposition \ref{thm_t=0}. Assume also that $p_g\leq 11$ and $C_0$ is not contained in $\bar{B}$. Then  there exists a rational map such that the image of $\bar{B}$ in $\pp^2$ is a curve of degree 14 with $12-p_g$ points of multiplicity 4 and one point of multiplicity 6 as unique essential singularities. For $r=2$, the singular point of multiplicity 6 is infinitely near to, at least, a point of multiplicity 4.
\end{cor}
\begin{proof}
From  Proposition \ref{thm_t=0} and Remark \ref{rem_r=3}, up to an elementary transformation of $F_r$ centered in a quadruple point, we can assume that $r=0$ or $2$.

 If $r=0$, let $f_1$ and $f_2$ be the two rulings of $\mathbb{F}_0$. Then the image of $\bar{B}$ in $\mathbb{F}_0$ is a curve $\bar{B}_{\mathbb{F}_0} \in \mid 8f_1 + 10f_2\mid$. Since $p_g\leq 11$, there exists at least one point of multiplicity 4. We blow up one of these points of $\mathbb{F}_0$ and obtain a new line, a $(-1)$-curve; next blow down the two fibres passing through the point. We obtain two singularities of multiplicity 4 and 6; the image of $\bar{B}_{\mathbb{F}_0}$ meets the new line in $6+ 4 +4=14$ points. In sum, there exists a birational map $\mathbb{F}_0 \dashrightarrow \mathbb{P}^2$ such that the image of $\bar{B}_{\mathbb{F}_0}$
 is a curve of degree 14 with two (distinct) points of multiplicity $6$ and $4$ plus $11-p_g$ points of multiplicity $4$. Note that some of the essential singular points are possibly infinitely near.

 Similarly, if $r=2$, then $\bar{B}_{\mathbb{F}_2}\in |8C_0 + 18f|$. Since by assumption $C_0$ is not contained in $\bar{B}$, as before, there exists at least one point of multiplicity 4 not contained in the infinity section, so there exists a birational map $\mathbb{F}_2 \dashrightarrow \mathbb{P}^2$ such that the image of $\bar{B}_{\mathbb{F}_2}$ in $\pp^2$ is a curve of degree 14 with a point of multiplicity $(6,4)$ plus $11-p_g$ points of multiplicity $4$. As in the case $r=0$ we note that apart from the singular point of multiplicity $(6,4)$, some of the other essential singularities are possibly infinitely near.
 \end{proof}


\

\section{Surfaces of type $(IIb)$}

\

We recall that $S$ is a surface of type $(IIb)$ if the canonical map factors through an involution with $t=2$. We can then  write the  branch curve of the double cover $V\to W$ as  $2L\equiv B'=B +A_1 + A_2$, where $A_1, A_2$ are $(-2)$-curves.

From Proposition \ref{numeros2} we have $h^0(3K_W + B)=p_g +2$, so if $P$ and $\bar{B}$ are as in Proposition \ref{prop_minimalmodel}, the effective divisor $3K_P+\bar{B}$ is nef.

Also, from Proposition \ref{numeros2}:
\begin{lem}\label{lem_t=2}
Let $S$ be a minimal surface with $K_S^2=2p_g$, $q=0$ and $p_g\geq 5$, and $i$ an involution on $S$ such that $t=2$. Then
$K_P(\bar{B}/2)=1-p_g-K_P^2$ and $(\bar{B}/2)^2=\bar{L}^2+1=3p_g-2+K_P^2$;
\end{lem}


Throughout this section we will prove the following theorem:

\begin{thm}\label{thm_t=2}
Let $S$ be a minimal surface with $K_S^2=2p_g$, $q=0$ and $p_g\geq 5$, such that the canonical map factors through an involution with $t=2$. Then $p_g\leq 8$ and one of the following occurs:
\begin{enumerate}
    \item $S$ is the minimal resolution of a double cover of a weak Del Pezzo surface $T$ of degree $p_g+1$ branched on an effective divisor in $|-4K_T|$ having exactly two (3,3)-points as essential singularities.

    \item  $S$ is the minimal resolution of a double cover of $\mathbb{F}_r$, $r\leq 2$,
    whose branch curve is the union of a curve in $|8C_0 + (9+4r)f)|$ with a fibre. The curve has $8-p_g$ singular points (possibly infinitely near) of multiplicity 4 and another of type $(4,4)$ and the fibre is tangent to the curve at the $(4,4)$-point.
        \end{enumerate}
\end{thm}
\begin{proof}
We divide the proof into steps.\\

\textbf{Step 1:} \textit{With the usual notation  $K_P^2 =  p_g-2$ or $K^2_P= p_g-3.$}

From  Proposition \ref{numeros2} $h^0(2K_W+L)=1$ and also
$0\leq (3K_P + \bar{B})(2K_P + \bar{L})= K_P^2 + 3 - p_g,$
therefore $K^2_P\geq p_g-3$. On the other hand, by the Index theorem $K_P^2(\bar{B}/2)^2\leq (K_P(\bar{B}/2))^2$ and hence, from Lemma \ref{lem_t=2}, we get $K_P^2\leq \frac{(p_g-1)^2}{p_g}$ and the assertion follows.

\


\textbf{Step 2:}
\textit{If $4K_P+ \bar{B}$ is not nef, there exists a birational morphism $p_1:P \to P_1$ such the divisor $4K_{P_1} + B_{P_1}$ is nef.}

From Proposition \ref{numeros2} and Step 1, $h^0(4K_P + \bar{B}) > 0$. So, if $4K_P + \bar{B}$ is not nef, there is an irreducible curve $E_1$ such that $E_1(4K_P + \bar{B})< 0$ and so as in Remark \ref{inv_t=0}, we can see that  $E_1$ is a $(-1)$-curve with $E_1(3K_P+\bar{B})=0$
 and so $E_1\bar{B}=3$. Since $\bar{B}'$ is an even divisor, it is clear that $E_1(C_1+C_2)>0$ and it is an odd number. Since $(E_1 + C_1 + C_2)(3K_P+\bar{B})=0$, $(E_1+C_1 +C_2)^2< 0$, by the Index theorem. As $(E_1+C_1+C_2)^2=-1 + 2E_1(C_1+C_2)+ -4$ the only possibility is $E(C_1+C_2)=1$.
  If, say, $E_1C_1=1$, when $E_1$ is contracted the image of $C_1$ is a $(-1)$-curve that is in the branch locus and intersects the image of $\bar{B}$ at a triple point. So at this point, the image of $4K_P+\bar{B}$ is not nef any more. Therefore, we have to contract the image of $C_1$ also, so the inductive step consists in contracting twice. If necessary, we can repeat the same argument for another $(-1)$-curve $E_2$ with $E_2C_2=1$, obtaining the result.

\

To continue with the proof, we analyse the two values of $K_P^2$.

\

\textbf{Step 3:}
\textit{If $K_P^2=p_g-3$, then $S$ is the minimal resolution of a double cover of a weak Del Pezzo surface $T$ of degree $p_g+1$ branched on a divisor in $|-4K_T|$ having two (3,3)-points.}

Let $p_1:P \to P_1$ be the birational morphism such that  $4K_{P_1} + B_{P_1}$ is nef.  If $s$  is the number of $(-1)$-curves contracted by $p_1$, from Lemma \ref{lem_t=2},  $(4K_P + \bar{B})^2=-4$ and one has $s\geq 4$. Besides, note that $K_{P_1}^2=K_P^2+s$ and $(B_{P_1}/2)^2=(\bar{B}/2)^2+\frac{9}{4}s$ by Lemma \ref{lem_t=2}, so
\begin{equation}\label{eq_step_3}
0\leq (2K_{P_1}+B_{P_1})(4K_{P_1}+B_{P_1})=4-s
\end{equation}
hence, $s=4$; since $4K_{P_1}+B_{P_1}$ is an effective divisor and by the Index theorem we have that $2K_{P_1} + B_{P_1}/2$ is a trivial divisor. Therefore $-K_{P_1}=K_{P_1} + B_{P_1}/2$ gives $-K_{P_1}$ nef and big, so  $P_1$ is a weak Del Pezzo surface of degree $K_{P_1}^2=p_g +1$.
Finally, let us analyse the image of $C_1$ and $C_2$ in $P_1$. As we have seen, $C_1$ and $C_2$ are $(-2)$-curves in $P$, however, if $E_1$ is a $(-1)$-curve contracted by $p_1$, by Step 2, we can suppose that $E_1C_1=1$ and $C_1$ becomes a $(-1)$-curve, whose intersection with the image of $\bar{B}$ is equal to 3 and it will be contracted as well. Since $s=4$, we can repeat the same argument for another $(-1)$-curve $E_2$ with $E_2C_2=1$, obtaining the two singular $(3,3)$-points.

\

\begin{remark}
{\rm Notice that  Theorem 4.2 of \cite{ba-pig} gives the same result as in Step 3 for the case $p_g=4$.}
\end{remark}

\textbf{Step 4:}
\textit{If $K_P^2=p_g-2$, then $|4K_P+\bar{B}|$ is a rational pencil without base points.}

Keeping the same notation as in the proof of Step 3, and using the Index theorem we obtain that $K_{P_1}^2(B_{P_1}/2)^2\leq (K_{P_1}(B_{P_1}/2))^2$, which implies $s\leq \frac{4}{p_g+2}$ by Lemma \ref{lem_t=2}; hence $s=0$ and we conclude that $4K_P+\bar{B}$ is nef. From Proposition \ref{numeros2}, we see that $(4K_P+\bar{B})^2=0$; besides, from Lemma \ref{lem_t=2}, we have $K_P(4K_P+\bar{B})=-2$. So, applying Lemma \ref{lem-marga}, the result follows.

\

\textbf{Step 5:}
\textit{If $K_P^2=p_g-2$, then $S$ is the minimal resolution of a double cover of $\mathbb{F}_r$, $r\leq 2$,
    whose branch curve is the union of a curve in $|8C_0 + (9+4r)f)|$ with a fibre. The curve has $8-p_g$ singular points of multiplicity 4 and another of type $(4,4)$ and the fibre is tangent to the curve at the $(4,4)$-point.}

From Step 4,
$|4K_P + \bar{B}|$ is a genus 0 pencil without base points, hence $P\neq \pp^2$ and contracting $10-p_g$  exceptional curves, we get a birational morphism
$\gamma: P \to \mathbb{F}_r$.

Let $f$ be a fibre of $\mathbb{F}_r$ and $C_0$ a section with $C_0^2=-r$. Write $\bar{B}=\gamma^*(aC_o +bf)-\sum c_iE_i$, since
$$4K_P+\bar{B}=\gamma^*((a-8)C_0 + [b-4(2+r)]f)+\sum (4-c_i)E_i=\gamma^*(f)$$

we obtain $\bar{B}=\gamma^*(8C_0 + (9+4r)f) - 4\sum E_i$.

Also, from $0\leq \gamma^*(C_0)(2K_P+\bar{B})$, we have $r\leq 2$.

 By Proposition \ref{prop_minimalmodel}, we know that $C_1$ and $C_2$ are $(-2)$-curves on $P$ and hence $C_i(4K_P + \bar{B})=0$, so they are contained in the fibres. More precisely, since $C_i(2K_P+\bar{L})=-1$, we can write $2K_P+\bar{L}=D + C_1 + C_2$ where $D$ is an effective divisor with $h^0(D)=1$. Hence, $4K_P + \bar{B}= 2D + C_1 + C_2$, so $C_1$ and $C_2$ are in the same fibre. By easy calculations we have $D^2=-1$ and $K_PD=-1$. Then it is easy to see that contracting $D$ and then, say, the image of $C_1$, we obtain a singularity of multiplicity $(4,4)$ of the image of $\bar{B}$, such that the fibre passing through this point is contained in the branch locus and it is tangent to $\bar{B}$ at the $(4,4)$-point. Finally, since there are $10-p_g$ singular points of multiplicity 4, then $p_g\leq 8$.
\end{proof}

To end this section we make several remarks.
\begin{remark}
{\rm If $S$ is a surface as in Step 5,
 we can  proceed as in the proof of Corollary \ref{cor_t=0}. First, by Step 5 we have $r\leq 2$, so we can suppose that $r=0,2$. If $r=0$, there exists a birational map $\mathbb{F}_0 \dashrightarrow \mathbb{P}^2$ such that we can see the image of the branch $B'$ on $\pp^2$ as a curve of degree $14$ of type $C +  l$, where $l$ is a line and $C$ is a curve of degree 13 having two (distinct) points $P_1$ and $P_2$ of multiplicity $5$ and $(4,4)$ respectively at the intersection with $l$, plus $8-p_g$ points of multiplicity $4$, and no further essential singularities. For the case $r=2$ it is easily seen that the infinity section $C_0$ cannot contain the $(4,4)$-point. Then with the same procedure the case $r=2$ can be expressed as  a degeneration of the case $r=0$. The degeneration consists of having $P_1$ infinitely near to $P_2$. Note also that $l$ is tangent to $C$ at $P_2$.}
 \end{remark}


\

\section{Surfaces of type $(IIc)$}

\

Finally, we are going to study surfaces with $K_S^2=2p_g$, $q=0$ and $p_g\geq 5$, such that the canonical map factors through an involution with $t=4$.

For these surfaces  we have that $h^0(2K_W+L)=0$. Thus the bicanonical map of $S$ is composed with $i$, hence not birational.   Since  by hypothesis $K_S^2\geq 10$, using (\cite{re}, Proposition 3), $S$ has a pencil of curves of genus $2$, necessarily rational because $q=0$.
We remark that the existence of the genus $2$ pencil can be also checked directly by considering the linear system $|3K_W+B|$.


\begin{remark}
{\rm It is easy to see that the rational pencil of curves of genus 2 is unique. Otherwise, let $|G_1|$ and $|G_2|$ be two pencils of genus 2 without base points; since $G_1G_2\geq 2$, $(G_1+G_2)^2\geq 4$. Since $K_S(G_1+G_2)=4$, we obtain $K_S^2(G_1+G_2)^2-(K_S(G_1+G_2))^2 >0$, a contradiction to  the Index theorem, because $K_S^2\geq 10$.}
\end{remark}

\begin{remark}\label{rem-horikawa2}
{\rm Recall that if a surface has a pencil of genus $2$, there exists a map of degree $2$
onto a ruled surface, mapping each genus 2 fibre by its canonical map onto a fibre of the ruling. Horikawa in \cite{ho2} proved that with elementary transformations it is possible to obtain a minimal model whose branch locus
has only singularities of the following types:
$(0)$, $(I_k)$, $(II_k)$, $(III_k)$, $(IV_k)$ and $(V)$ (in Horikawa's notation).

In what follows, we will say that a rational fibre of Horikawa's model is {\it a singular fibre of type} $(I_k)$, $(II_k)$, $(III_k)$, $(IV_k)$ or $(V)$ if the branch locus has the corresponding singularity or singularities on the fibre.
We will also say that a genus 2 fibre is {\it of type} $(I_k)$, $(II_k)$, $(III_k)$, $(IV_k)$ or $(V)$ if its corresponding rational fibre is of this type.
}
\end{remark}

The next result is well known but for completeness we include its proof.
\begin{lem}\label{lem-genus2}
Let $S$ be a minimal algebraic surface of general type  with $ p_g\geq 4$, $q=0$  and canonical map not composed with a pencil. If $S$ has a pencil of genus $2$, then the canonical map has degree $2$.
\end{lem}
\begin{proof}
Let $|G|$ be the genus 2 pencil on $S$ and note that $K_S|_G\simeq \omega_G$. Since   $\omega_G$ is a $g^1_2$ on $G$,
$\phi_{K_S}$ has even degree.

Since $h^0(S,K_S)\geq 4$ and $h^0(G,\omega_G)=2$, from the long exact sequence:
$$ 0 \to \mathcal{O}_S(K_S-G) \to \mathcal{O}_S(K_S) \to \mathcal{O}_S(K_S)\mid_G \to 0,$$
we have that $h^0(S,K_S-G)\geq 2$. Therefore, $\phi_{K_S}$ separates the fibres and we obtain the result.
\end{proof}

\begin{prop}\label{prop-genus2} The pencil of genus $2$ in $S$ is the pull-back of a ruling of  the canonical image $\Sigma$ of $S$.
Moreover, the essential singularities of the branch locus in Horikawa's model are of type: $(I_k)$, $(II_k)$, $(III_k)$, $(IV_k)$, with $k=1,2$, and $(V)$.
\end{prop}
{\bf Proof:}
Since $K_S|_G=\omega _G$ and  $|K_S|$ is not composed with a pencil, the image of a general element of $|G|$
 is a line, and so a ruling of $\Sigma$.
Using (\cite{ho2}, Theorem 3) we get
\begin{equation}\label{eq-horikawa}
4= \sum_k\{(2k-1)(\nu(I_k) + \nu(III_k)) + 2k(\nu(II_k)+\nu(IV_k))\} + \nu(V),
\end{equation}
where $\nu(*)$ denotes the number of singularities of type $(*)$. So we immediately obtain that $k=1,2$. \Qed

\begin{remark}\label{horikawa}
{\rm
Looking carefully at the resolution of the singular fibres as in Proposition \ref{prop-genus2}, we obtain that:
\begin{enumerate}
\item each singular fibre of type $(I_1)$ or $(III_1)$ or $(V)$ corresponds to one base point of $|K_S|$ and one isolated fixed point of the involution;
\item each singular fibre  of type $(I_2)$ or $(III_2)$ corresponds to a fixed component plus one base point of $|K_S|$, and three isolated fixed points of the involution;
\item  each singular fibre of type $(II_1)$ or $(IV_1)$ corresponds to a fixed component of $|K_S|$, and there are two isolated fixed points of the involution;
\item finally, each singular fibre of type $(II_2)$ or $(IV_2)$ corresponds to a fixed component plus two base points of $|K_S|$, and four isolated fixed points of the involution.
\end{enumerate}
We point out that all the fixed components of $|K_S|$  are $(-2)$-curves.}
\end{remark}

Using double or bidouble covers, it is not difficult to find examples, for instance,

\begin{ex}
In $\mathbb{F}_0$, let $f_1$ and $f_2$ denote general fibres of each ruling.  Consider the bidouble cover $\pi: S \to \mathbb{F}_0$ with smooth branch curves $D_1\in |f_1+f_2|, D_2\in |f_1 + 3f_2|$ and $D_3\in |3f_1+(2a+1)f_2|$, $a\geq 2$.  Using the bidouble cover formulas (see \cite{cata}, \cite{pa}), the surface $S$ has the invariants $K^2=4a+2$, $p_g=2a+1$ and $q=0$. Now, we analyse the double cover as a composition of two double covers. First, the double cover with branch curve $D_1 + D_2$, this is $Y \to \mathbb{F}_0$, where $Y$ is a rational surface with four $(-2)$-curves coming from the intersection points of $D_1$ and $D_2$. The linear system $|\pi^*f_2|$ is the pencil of genus $2$, whose fibres  in the general case will be of type $(I_1)$.  It is easy to see that a mild degeneration of the construction will yield   fibres of type  $(III_1)$ or $V$.
\end{ex}

\

\

\end{document}